\documentclass[11pt]{article}

\usepackage{subfigure}
\usepackage{color}
\usepackage[cmex10]{amsmath}
\usepackage{mathrsfs,amssymb,amsmath}
\usepackage{cases}
\usepackage{graphicx}
\usepackage{caption}

\usepackage{epstopdf}
\usepackage{hyperref}

\usepackage{float}
\usepackage{placeins}

\newcommand{\R}{{\mathbb R}}

\def\dref#1{(\ref{#1})}

\usepackage{algorithm} 
\usepackage{algorithmic} 

\newtheorem{corollary}{Corollary}[section]
\newtheorem{remark}{Remark}[section]
\newtheorem{lemma}{Lemma}[section]
\newtheorem{assumption}{Assumption}[section]

\newtheorem{notation}{Notations}[section]
\newtheorem{theorem}{Theorem}[section]
\newtheorem{definition}{Definition}[section]

\newenvironment{proof}{{\bf {Proof.}}}{\hfill $\square$}
\allowdisplaybreaks[4]

\numberwithin{equation}{section}

\def\dref#1{(\ref{#1})}

\def\pt{\partial}

\def\ra{\rightarrow}

\def\s{\subseteq}

\def\e{\varepsilon}

\def\ol{\overline}

\def\vp{\varphi}

\def\bf{\textbf}
\def\pt{\partial}

\def\om{\omega}
\def\Om{\Omega}
\def\la{\lambda}
\def\al{\alpha}
\def\be{\beta}
\def\de{\delta}

\def\Ga{\Gamma}
\def\La{\Lambda}

\def\iy{\infty}

\def\f{\frac}

\def\se{\setminus}

\def\Lra{\Leftrightarrow}

\def\Span{{\rm{span}}}

\def\df{\mathrm d}

\def\wh{\widehat}

\def\essinf{\operatorname*{ess\ \! inf}}

\def\diag{{\rm diag}}

\def\mcL{\mathcal{L}}

\def\mcO{\mathcal{O}}
\def\mcA{\mathcal{A}}

	\DeclareMathOperator{\diam}{diam}

	\DeclareMathOperator{\Div}{div}

	\DeclareMathOperator{\dist}{dist}

	\DeclareMathOperator{\supp}{{supp}}

	\newcommand{\N}{\mathbb N}

\begin{document}

\title{{\bf   {\bf    Some Key Properties of Eigenfunctions Linked to Degenerate Elliptic Differential Operators}}\footnote{\small This work was carried out with the support of the
National Natural Science Foundation of China under grant  nos. 12131008 and U23B2033, and
National Key R\&D Program of China under grant no. 2024YFA1013101.}}

\author{ Dong-Hui Yang$^{a}$, Bao-Zhu Guo$^{b}$\footnote{\small
The corresponding author. Email: bzguo@is.ac.cn}
\\
$^a${\it School of Mathematics and Statistics, Central South University}\\
			{\it Changsha 410075, P.R.China}\\
$^b${\it Academy of Mathematics and Systems Science, Academia Sinica, Beijing 100190, China}}

\date{}

\maketitle{}

\begin{abstract}
    In this study, we address the eigenvalue problem given by:
\begin{equation*}
\begin{cases}
-\Div (w\nabla u_i)=\la_iu_i &\text{in } \Om\subset \mathbb{R}^n,\\
u_i=0 &\text{on } \pt \Om,
\end{cases}
\end{equation*}
where $w > 0$ within $\Om$ and $w = 0$ on part of $\partial \Omega$. We establish Courant's nodal domain theorem for the corresponding degenerate elliptic differential operator $\mathcal{A}$.
Unlike uniformly elliptic operators, degenerate cases often result in the loss of many advantageous properties. Despite this, we show that the essential property that the set $\{\rho \in L^\infty(\Omega) \colon \mathcal{A} + \rho \text{ has simple eigenvalues}\}$ forms a residual subset within $(L^\infty(\Omega), |\cdot|_\infty)$ still holds for the degenerate elliptic differential operator $\mathcal{A}$.

\vspace{0.2cm}

	
\vspace{0.2cm}
		

	\end{abstract}

	 \section{Introduction}\label{S1}
	
	  Courant's nodal domain theorem was initially introduced in \cite{Hilbert}. This theorem states that the $k$-th eigenfunction of the Laplacian operator $\Delta$, subject to Dirichlet boundary conditions, has at most $k$ nodal domains.  Since its inception, the theorem has been the subject of extensive research within the framework of uniformly elliptic and parabolic equations, as well as concerning the $p$-Laplacian operator; for instance, see \cite{Drabek,Grebenkov,Helffer,Henrot,Sanayei}. However, to the best of our knowledge, no existing results address Courant's nodal domain theorem in the context of degenerate elliptic operators. Proving this theorem is far from trivial; it necessitates the use of tools from spectral theory \cite{Grebenkov,Henrot,Kato}, Sard's theorem \cite{Sard}, and the unique continuation property \cite{Aronszajin,Rousseau} for partial differential operators. These tools now play a pivotal role in the analysis of partial differential equations. It is important to note that the unique continuation property may not hold for certain specialized classes of partial differential operators \cite{Martio}. At present, we are unaware of any method to prove Courant's nodal domain theorem without relying on this property.

 Research on Courant's nodal domain theorem has predominantly advanced along two main avenues: one delves into the generic characteristics of eigenfunctions, while the other scrutinizes the Hausdorff measure of their nodal sets.
The generic properties of eigenfunctions associated with partial differential equations have been explored in a series of studies, including \cite{Albert,Albert1,Grebenkov,Uhlenbeck}. Concurrently, the Hausdorff measure of nodal sets for eigenfunctions of partial differential operators has been the subject of research in works such as \cite{Grebenkov,Han,Logunov,Yau}.
Degenerate partial differential equations have been extensively examined in the existing literature; for illustrative examples, refer to \cite{Cavalheiro,CS2,FF,Fabes,GC,Guo,Heinonen,Wu1}. Nevertheless, to extend Courant's nodal domain theorem to encompass degenerate differential equations, it is imperative to utilize tools from spectral theory, Sard's theorem, and the unique continuation property. Acquiring the unique continuation property for degenerate partial differential operators is itself a formidable challenge; for further details, see \cite{Guo,Wu1}. When shifting the focus from uniformly elliptic operators to degenerate elliptic operators, a multitude of properties experience alterations. For example, whereas the $k$-th eigenfunction $\Phi_k\ (k\in\N)$ typically exhibits $C^\infty(\overline{\Omega})$ regularity in the uniformly elliptic scenario, such smoothness generally does not persist for degenerate elliptic operators; indeed, $\Phi_k$ may even lack classical first-order partial derivatives. As a result, ascertaining the generic properties of eigenfunctions for degenerate elliptic equations presents a significant hurdle.
 
 As delineated in \cite{Albert1}, the generic properties of eigenfunctions associated with uniformly elliptic operators $L$ can be systematically categorized into no fewer than four distinct scenarios: a) every critical point of the eigenfunction $\Phi_k$ exhibits Morse characteristics and is non-nodal; b) critical points that are distinct from one another possess correspondingly distinct critical values; c) the value $0$ does not constitute a critical value, implying the absence of nodal critical points; d) the set $\{\rho\in C^\iy(M)\colon \text{all eigenvalues of } L+\rho \text{ are simple} \}$ forms a residual subset of $C^\iy(M)$. The preceding analysis elucidates that while properties a)-c) prove challenging to establish in degenerate cases, property d) retains its validity: this distinction underscores a principal contribution of the present work.
 
To study degenerate elliptic equations, it is crucial to introduce the concept of $A_p$ weights.
When $p = 1$, a locally integrable non-negative function $w(\cdot)$ is defined as an $A_1$ weight if there exists a constant $c=c(w,1) > 0$ such that for all cubes $K \subset \mathbb{R}^N$,
\begin{equation*}
\frac{1}{|K|}\int_K w(x) \, \mathrm{d}x \leq c \cdot \essinf_{K} w.
\end{equation*}
For $p \in (1, \infty)$, a locally integrable non-negative function $w(\cdot)$ is termed an $A_p$ weight if there exists a constant $c > 0$ such that for all cubes $K$ in $\mathbb{R}^N$,
\begin{equation}\label{01.18.2}
\left(\frac{1}{|K|}\int_K w(x) \, \mathrm{d}x\right)\left(\frac{1}{|K|}\int_K w(x)^{-\frac{1}{p - 1}} \, \mathrm{d}x\right)^{p - 1} \leq c.
\end{equation}
The infimum of the set of constants $c = c(w, p) > 0$ for which \eqref{01.18.2} holds is referred to as the $A_p \ (1\leq p<\iy)$ constant of $w(\cdot)$.
Specifically, taking $w(x)=|x-x_0|^\al \ (0<\al<2)$, we find that $w$ is an $A_{1+\f{2}{N}}$ weight \cite{Heinonen}.

 In this paper, we address the following eigenvalue problem:
\begin{equation}\label{01.08.1}
\begin{cases}
-\Div (w\nabla u_i)=\lambda_iu_i, &\text{in } \Omega,\\
u_i=0, &\text{on } \partial \Omega,
\end{cases}
\end{equation}
where $w\in C(\overline{\Omega})\cap C^{\infty}(\Omega)$ serves as an $A_2$ weight on $\Omega\subset\mathbb{R}^N$, satisfying
\begin{equation}\label{gbz5}
w>0 \text{ on } \Omega \quad \text{and} \quad w=0 \text{ on }\Gamma, 
\end{equation}
and Assumption \ref{06.27.A1}. 
Here, $\Omega$ represents a bounded smooth domain, and $\Gamma\subset \partial\Omega$ is either an open subset or a singleton set. For instance, when $N = 2$, we can define $w(x)=|x - x_0|^{\alpha}$ for $x=(x_1,x_2)\in\mathbb{R}^2$, where $|x - x_0|=\sqrt{(x_1 - x_1^0)^2+(x_2 - x_2^0)^2}$ and $x_0=(x_1^0, x_2^0)\in \partial\Omega$, with $\alpha\in (0,2)$ being a predetermined constant. Further examples are provided in \cite{Guo}. The symbol $\lambda_i\ (i\in\mathbb{N})$ denotes the $i$-th eigenvalue of the partial differential operator
\begin{equation}\label{06.25.1}
\mathcal{A} y=\Div(w \nabla y),
\end{equation}
and $u_i$ corresponds to the $i$-th eigenfunction of $\mathcal{A}$ associated with the $i$-th eigenvalue $\lambda_i$. It is crucial to emphasize that the findings presented in this paper are also applicable to partial differential operators of the following form:
\begin{equation*}
\mathcal{A} y=\Div(A\nabla y),
\end{equation*}
where $A(x)=(a_{ij}(x))_{i,j = 1,\cdots,N}$ with $a_{ij}\in C(\overline{\Omega})\cap C^{\infty}(\Omega)$ for all $i,j = 1,\cdots,N$. Additionally, there exists a constant $\Lambda>0$ such that
\begin{equation*}
\Lambda^{-1}|\xi|^2w(x)\leq \sum_{i,j = 1}^N a_{ij}(x)\xi_i\xi_j\leq \Lambda |\xi|^2 w(x) \quad \text{for all } \xi\in\mathbb{R}^N \text{ and for } x\in \Omega \text{ almost everywhere},
\end{equation*}
where $w(\cdot)$ is defined as previously mentioned.

We proceed as follows. In Section \ref{S2}, we delineate the solution space pertinent to equation \eqref{01.08.1}. Subsequently, in Section \ref{S3}, we establish Courant's nodal domain theorem. In Section \ref{S4}, we demonstrate that the set $\{\rho \in L^\infty(\Omega) \colon \mathcal{A} + \rho \text{ possesses simple eigenvalues}\}$ contains a residual subset within the space $(L^\infty(\Omega), |\cdot|_\infty)$.

The proofs presented in this paper predominantly draw upon functional analysis principles; nevertheless, numerous specifics are tailored to the degenerate scenario. To ensure thoroughness, we furnish comprehensive proofs, even when they overlap with those found in existing literature.

\section{Solution spaces}\label{S2}

 In this section, we introduce the solution spaces for equation \eqref{01.08.1}. These spaces are known as weighted Sobolev spaces (see \cite{GC,Heinonen}) and differ from classical Sobolev spaces for uniformly elliptic partial differential equations, as the measures $w\df x$ may not be invariant under transformations.
We define
\begin{equation*}
H^1(\Om;w)=\left\{u\in L^2(\Om)\colon \int_\Om |\nabla u|^2w\df x+\int_\Om u^2 \df x<+\infty\right\},
\end{equation*}
with   inner product
\begin{equation*}
(u,v)_{H^1(\Om;w)}=\int_\Om (\nabla u\cdot\nabla v)w\df x+\int_\Om uv\df x,
\end{equation*}
and norm
\begin{equation*}
\|u\|_{H^1(\Om;w)}=\left(\int_\Om |\nabla u|^2w\df x+\int_\Om u^2\df x\right)^{\frac{1}{2}}.
\end{equation*}
Let
\begin{equation*}
H_0^1(\Om;w)=\text{the completion of } C_0^\infty(\Om) \text{ in } (H^1(\Om;w), \|\cdot\|_{H^1(\Om;w)}).
\end{equation*}
Moreover, we define
\begin{equation*}
H^2(\Om;w)=\left\{u\in H^1(\Om;w)\colon \|u\|_{H^1(\Om;w)}^2+\int_\Om |\mcA u|^2\df x<\infty\right\},
\end{equation*}
with inner product
\begin{equation*}
(u,v)_{H^2(\Om;w)}=(u,v)_{H^1(\Om;w)}+\int_\Om (\mcA u)(\mcA v)\df x,
\end{equation*}
and norm
\begin{equation*}
\|u\|_{H^2(\Om;w)}=\left(\|u\|_{H^1(\Om;w)}^2+\int_\Om |\mcA u|^2\df x\right)^{\frac{1}{2}}.
\end{equation*}
The following lemma is a straightforward conclusion from \cite{GC,Heinonen}.
\begin{lemma}\label{06.04.L1}
The spaces $(H^1(\Om;w), (\cdot, \cdot)_{H^1(\Om;w)})$ and $(H^2(\Om;w), (\cdot, \cdot)_{H^2(\Om;w)})$ are Hilbert spaces.
\end{lemma}
We now present the assumptions for the space $H_0^1(\Omega; w)$, which play a central role in this paper.

\begin{assumption}\label{06.27.A1}
1)~ The  Poincar\'e inequality holds in $H_0^1(\Om;w)$, i.e., there exists a constant $C>0$ such that
\begin{equation*}
\int_\Om u^2\df x\leq C\int_\Om |\nabla u|^2w\df x
\end{equation*}
for all $u\in H_0^1(\Om;w)$;

2)~ The embedding  $H_0^1(\Om;w)\hookrightarrow W^{1,1}(\Om)$ is continuous;

3)~ The embedding $H_0^1(\Om;w)\hookrightarrow L^2(\Om)$ is compact.
\end{assumption}

\begin{remark}\label{06.27.R1}
From Assumption \ref{06.27.A1},  we obtain that the norm
\begin{equation*}
\|u\|_{H^1(\Om;w)}=\left(\int_\Om |\nabla u|^2w\df x\right)^{\frac{1}{2}}
\end{equation*}
is an equivalent norm on $(H_0^1(\Om;w), \|\cdot\|_{H^1(\Om;w)})$. Here and in what follows, we use this norm on $H_0^1(\Om;w)$.
We note that the Poincar\'e inequality holds for many degenerate elliptic equations, such as the classical Grushin operator on a cube. These inequalities can often be derived from Hardy's inequalities for degenerate elliptic equations. The  Case 2) in Assumption \ref{06.27.A1} implies that the trace of $u\in H_0^1(\Om;w)$ is the classical Sobolev trace.
\end{remark}

As an example, for the special case $w=|x-x_0|^\al$ as described in Section \ref{S1}, we prove that $H_0^1(\Om;w)$ satisfies Assumption \ref{06.04.L2}.
\begin{lemma}\label{06.04.L2}
Let $\al\in (0,2)$. Then there exists a constant $C>0$, depending only on $\al$, such that
\begin{equation*}
\int_\Om |x-x_0|^{\al-2}u^2\df x\leq C\int_\Om |x-x_0|^\al |\nabla u|^2\df x
\end{equation*}
for all $u\in H_0^1(\Om;w)$.
\end{lemma}
\begin{proof}
Let $z\in H_0^1(\Om;w)$. For any  $\e>0$, since $u\in W^{1,2}(\Om\setminus B_\e(x_0))$, we have
\begin{equation*}
\begin{split}
&2\int_{\Om\setminus B_\e(x_0)}|x-x_0|^{\al-2}z\left[(x-x_0)\cdot \nabla z\right]\df x\\
&=\int_{\Om\setminus B_\e(x_0)}|x-x_0|^{\al-2}(x-x_0)\cdot \nabla (z^2)\df x\\
&=\int_{\Om\cap \partial B_\e(x_0)}z^2|x-x_0|^{\al-2}(x-x_0)\cdot\nu\df S-\al\int_{\Om\setminus B_\e(x_0)}|x-x_0|^{\al-2}z^2\df x,
\end{split}
\end{equation*}
where $\nu$ is the outer normal on $\partial B_\e(x_0)$.
Then, from $(x-x_0)\cdot \nu=-|x-x_0|\leq 0$, we obtain
\begin{equation*}
\begin{split}
\al\int_{\Om\setminus B_\e(x_0)}|x-x_0|^{\al-2}z^2\df x
&\leq -2\int_{\Om\setminus B_\e(x_0)}|x-x_0|^{\al-2}z\left[(x-x_0)\cdot \nabla z\right]\df x\\
&\leq 2\int_{\Om\setminus B_\e(x_0)}\left(|x-x_0|^{\frac{\al}{2}-1}|z|\right)\left(|x-x_0|^{\frac{\al}{2}}|\nabla z|\right)\df x\\
&\leq 2\left(\int_{\Om\setminus B_\e (x_0)}|x-x_0|^{\al-2}z^2\df x\right)^{\frac{1}{2}}\left(\int_{\Om\setminus B_\e(x_0)}|x-x_0|^\al |\nabla z|^2\df x\right)^{\frac{1}{2}},
\end{split}
\end{equation*}
and hence
\begin{equation*}
\begin{split}
\int_{\Om\setminus B_\e(x_0)}|x-x_0|^{\al-2}z^2\df x
&\leq \left(\frac{2}{\al}\right)^2\int_{\Om\setminus B_\e(x_0)}|x-x_0|^\al |\nabla z|^2\df x\\
&\leq \left(\frac{2}{\al}\right)^2\int_{\Om}|x-x_0|^\al |\nabla z|^2\df x.
\end{split}
\end{equation*}
Letting $\e\to 0^+$ leads to the required result.
\end{proof}

\begin{remark}\label{06.04.R1}
From Lemma \ref{06.04.L2}, we obtain
\begin{equation*}
\int_\Om u^2\df x\leq C\int_\Om |\nabla u|^2w\df x,
\end{equation*}
which is the classical Poincar\'e inequality. From this, we get the following norm
\begin{equation*}
\|u\|_{H^1(\Om;w)}=\left(\int_\Om |\nabla u|^2w\df x\right)^{\frac{1}{2}}
\end{equation*}
is the equivalent norm of $(H_0^1(\Om;w),\|\cdot\|_{H^1(\Om;w)})$, which we shall use later.
\end{remark}

\begin{lemma}\label{06.04.L4}
The embedding $H^1(\Om;w)\hookrightarrow W^{1,1}(\Om)$ is continuous.
\end{lemma}
\begin{proof}
Denote $m=\diam\Om+1$. From
\begin{equation*}
\begin{split}
\int_\Om |\nabla u|\df x
&\leq \left(\int_\Om |x-x_0|^{-\al}\df x\right)^{\frac{1}{2}}\left(\int_\Om |x-x_0|^{\al}|\nabla u|^2\df x\right)^{\frac{1}{2}}\\
&\leq \left(\int_{B_m(x_0)}|x-x_0|^{-\al}\df x\right)^{\frac{1}{2}}\left(\int_\Om |\nabla u|^2w\df x\right)^{\frac{1}{2}}\leq C_m\|u\|_{H^1(\Om;w)},
\end{split}
\end{equation*}
where the constant $C_m$ depends only on $m$,
we complete the proof of the lemma.
\end{proof}
\begin{remark}\label{06.04.R2}
From Lemma \ref{06.04.L4}, we get the trace of $u\in H_0^1(\Om;w)$ on $\partial\Om$ is the classical Sobolev trace zero.
\end{remark}
\begin{lemma}\label{06.04.L3}
The embedding $H_0^1(\Om;w)\hookrightarrow L^2(\Om)$ is compact.
\end{lemma}
\begin{proof}
Let $\e>0$ be sufficiently small. Assume $\{u_n\}_{n\in\N}\subset H_0^1(\Om;w)$ is bounded, i.e., there exists $M>0$ such that $\|u_n\|_{H^1(\Om;w)}\leq M$ for all $n\in\N$. Then there exists a subsequence of $\{u_n\}_{n\in\N}$, still denoted by itself, and $u_0\in H_0^1(\Om;w)$ such that
\begin{equation*}
u_n\rightharpoonup u_0 \mbox{ weakly in } H_0^1(\Om;w).
\end{equation*}
Replace $u_n$ by $u_n-u_0$ for all $n\in\N$, then $u_n\rightharpoonup 0$ weakly in $H_0^1(\Om;w)$. In the following, we only need to show that $u_n\to 0$ strongly in $L^2(\Om)$ by abstract subsequence.
Now, on one hand, from Lemma \ref{06.04.L2}, we  have
\begin{equation*}
\int_{\Om\cap B_\e(x_0)}u^2\df x=\int_{\Om\cap B_\e(x_0)}|x-x_0|^{2-\al}|x-x_0|^{\al-2}u^2\df x\leq C\e^{2-\al}\int_{\Om\setminus B_\e(x_0)}|\nabla u|^2w\df x\leq C\e^{2-\al}M^2.
\end{equation*}
On the other hand, since $\Om_\e\subset \Om\setminus B_\e(x_0)$ satisfies $\Om_\e \supset \Om\setminus B_{2\e}(x_0)$  and the cone property, we have $H^1(\Om_\e;\al)=H^1(\Om_\e)\hookrightarrow L^2(\Om_\e)$ is compact (\cite[Theorem 6.3, p.168]{Adams}). Then there exists $u_{n(\e)}$ such that
\begin{equation*}
\|u_{n(\e)}\|_{L^2(\Om_\e)}<\e.
\end{equation*}
These imply that
\begin{equation*}
\|u_{n(\e)}\|_{L^2(\Om)}\leq \e +C\e^{2-\al}M^2.
\end{equation*}
This completes the  proof of the lemma.
\end{proof}

\begin{definition}\label{06.04.D1}
Let $w(\cdot)$ be an $A_2$ weight. We call $u\in H_0^1(\Om;w)\cap H^2(\Om;w)$ a {\it weak solution} of
\begin{equation}\label{06.04.1}
\begin{cases}
-\mcA u+Vu=f, &\mbox{in }\Om, \\
u=0, &\mbox{on } \partial\Om,
\end{cases}
\end{equation}
if for every  $\psi\in C_0^\infty(\Om)$, there hods 
\begin{equation*}
\int_\Om (\nabla u\cdot \nabla\psi)w\df x+\int_\Om Vu\psi\df x=\int_\Om f\psi\df x, 
\end{equation*}
where $f\in L^2(\Om)$,  and $0\leq V\in L^\infty(\Om)$.
\end{definition}

Note that if $w(\cdot)$ is an $A_2$ weight then $w^{-1}\in L^1(\Om)$.

\begin{lemma}\label{06.04.L5}
The degenerate elliptic equation \eqref{06.04.1}
has a unique weak solution. Moreover,  the following  estimate holds 
\begin{equation*}
\|\mcA u\|_{L^2(\Om)}+\|u\|_{H^1(\Om;w)}\leq C\|f\|_{L^2(\Om)},
\end{equation*}
where the constant $C>0$ depends only on $\Om$ and $\|V\|_{L^\infty(\Om)}$.
\end{lemma}
\begin{proof}
This is a conclusion of Lemma \ref{06.04.L2} (or Remark \ref{06.04.R1}) and the Lax-Milgram Theorem (see also
   \cite[Theorem 2.10, p. 367]{Cavalheiro}).
Multiplying $u$ on both sides of \eqref{06.04.1}, integrating on $\Om$, we have
\begin{equation*}
\int_\Om |\nabla u|^2w\df x+\int_\Om Vu^2\df x=\int_\Om fu\df x,
\end{equation*}
and by case 1) in Assumption \ref{06.27.A1}, 
\begin{equation}\label{06.27.1}
\|u\|_{H^1(\Om;w)}\leq C\|f\|{L^2(\Om)}
\end{equation}
  Finally,  
\begin{equation*}
\|\mcA u\|_{L^2(\Om)}\leq \|f\|_{L^2(\Om)}+\|V\|_{L^\infty(\Om)}\|u\|_{L^2(\Om)}\leq C\|f\|_{L^2(\Om)}
\end{equation*}
from case 1) in Assumption \ref{06.27.A1} and \eqref{06.27.1}. This completes the proof of the lemma.
\end{proof}

\section{Spectrum and Courant's nodal domain theorem}\label{S3}

Under Assumption \ref{06.27.A1}, we can determine the spectrum of the degenerate partial differential operator $\mcA$ defined by \dref{06.25.1}.
Note that $\mcA$  is a densely defined, self-adjoint operator. We define $\la\in\R$ as an {\it eigenvalue} of $\mcA$ if $\ker(\mcA+\la)\neq \{0\}$. The corresponding $\ker(\mcA+\la)$ is termed the {\it eigenspace} for $\la$, and any $u\in \ker(\mcA+\la)$ is referred to as an {\it eigenfunction} of $\mcA$ associated with the eigenvalue $\la$. We say that the partial differential operator $\mcA$ possesses {\it simple} eigenvalues if the eigenspace corresponding to the eigenvalue $\la$ is one-dimensional.
From Assumption \ref{06.27.A1}, we deduce that the partial differential operator $\mcA$ exhibits a discrete spectrum:
\begin{equation}\label{06.04.2}
0<\la_1\leq \la_2\leq \la_3\leq \cdots \ra +\iy.
\end{equation}

The following Lemma \ref{06.04.L6} is derived from   \cite[Theorem 2, Chapter 6, p. 336; and Theorems 4-7 in Appendix D,p.640-645]{Evans}. Although the proof of this lemma follows a similar approach to the referenced theorems, it includes slight modifications. For completeness, we provide a detailed, self-contained proof below. Furthermore, this lemma leads to the conclusion that $0<\la_1<\la_2$.

\begin{lemma}\label{06.04.L6}
The following results hold true: 

(i) 
\begin{equation}\label{gbz1} 
\la_1=\inf_{0\neq u\in H_0^1(\Om;w)}\f{\int_\Om (\nabla u\cdot \nabla u)w\df x}{\int_\Om u^2\df x}.
\end{equation}

(ii) Moreover, the infimum in \dref{gbz1}  is achieved by a function $u_1$, which is positive within $\Om$ and satisfies
\begin{equation*}
\begin{cases}
-\mcA u_1=\la_1 u_1, &\mbox{in }\Om,\\
u_1=0, &\mbox{on }\pt\Om.
\end{cases}
\end{equation*}

(iii) If $u\in H_0^1(\Om;w)$ is any weak solution of
\begin{equation*}
\begin{cases}
-\mcA u=\la_1u, &\mbox{in }\Om, \\
u=0, &\mbox{on }\pt\Om,
\end{cases}
\end{equation*}
then $u$ is a scalar multiple of $u_1$.
\end{lemma}
\begin{proof} The proof will be split into five steps. 

{\it Step 1}.  Let $u_i\ (i\in \N)$ denote the $i$-th eigenfunction of $\mcA$ corresponding to the $i$-th eigenvalue $\la_i$, and let $\{u_i\}$ form an orthonormal basis of $L^2(\Om)$ (\cite[Theorem 7, Appendix E,p. 645]{Evans}). From \eqref{01.08.1}, we have
\begin{equation}\label{06.04.3}
\int_\Om (\nabla u_k\cdot\nabla u_l)w\df x=\de_{kl}\la_k \mbox{ for } k,l\in \N,
\end{equation}
where $\de_{kl}$ is the Kronecker delta function, defined as $\de_{kl}=1$ for $k=l$ and $\de_{kl}=0$ for $k\neq l$.
We now demonstrate that $\{\la_k^{-\f{1}{2}}u_k\}_{k=1}^\iy$ forms an orthonormal basis of $H_0^1(\Om;w)$.
From \eqref{06.04.3}, it is straightforward to verify that $\{\la_k^{-\f{1}{2}}u_k\}_{k=1}^\iy$ is an orthonormal subset of $H_0^1(\Om;w)$. To prove it is a basis, assume by contradiction that there exists $0\neq u\in H_0^1(\Om;w)$ such that
\begin{equation*}
\int_\Om (\nabla u_k\cdot \nabla u)w\df x=0 \mbox{ for all }k\in\N.
\end{equation*}
Given that $\{u_k\}_{k\in\N}$ is an orthonormal basis of $L^2(\Om)$, for $u\in H_0^1(\Om;w)$, we can express
\begin{equation}\label{06.04.4}
u=\sum_{k=1}^\iy d_ku_k \mbox{ where }  d_k=(u, u_k)_{L^2(\Om)}, k\in\N.
\end{equation}
Then, by Definition  \ref{06.04.D1}, \eqref{01.08.1}, and \eqref{06.04.2}, we have
\begin{equation*}
0=(u,u_k)_{H^1(\Om;w)}=\int_\Om (\nabla u_k\cdot \nabla u)w\df x=\la_k\int_\Om u_ku\df x =\la_kd_k,
\end{equation*}
which implies $d_k=0$. This leads to $u=0$, a contradiction.

{\it Step 2: Let $u=\sum_{k=1}^\iy \mu_k\la_k^{-\f{1}{2}}u_k$ in $H_0^1(\Om;w)$.}  Combining this with \eqref{06.04.4}, we find $\mu_k=d_k\la_k^\f{1}{2}$ and
\begin{equation*}
\int_\Om (\nabla u\cdot\nabla u)w\df x=\sum_{k=1}^\iy \mu_k^2=\sum_{k=1}^\iy d_k^2\la_k\geq \la_1 \|u\|_{L^2(\Om)}^2,
\end{equation*}
which establishes (i).

{\it Step 3: Show that if $u\in H_0^1(\Om;w)$ and $\|u\|_{L^2(\Om)}=1$, then $u$ is a weak solution of
\begin{equation}\label{06.04.5}
\begin{cases}
-\mcA u=\la_1u, &\mbox{in }\Om,\\
u=0, &\mbox{on } \pt\Om,
\end{cases}
\end{equation}
if and only if
\begin{equation}\label{06.04.6}
\int_\Om (\nabla u\cdot \nabla u)w\df x=\la_1.
\end{equation}}

Clearly, \eqref{06.04.5} implies \eqref{06.04.6}.
Assume \eqref{06.04.6} holds. Write $d_k=(u,u_k)_{L^2(\Om)}$ as in \eqref{06.04.4}, then $\sum_{k=1}^\iy d_k^2=1$, and from \eqref{06.04.6} we get
\begin{equation*}
\sum_{k=1}^\iy d_k^2\la_1=\la_1=\sum_{k=1}^\iy d_k^2\la_k.
\end{equation*}
This implies $d_k=(u,u_k)=0$ for all $\la_k>\la_1$.
Since $\la_1$ has finite multiplicity, it follows that $u=\sum_{k=1}^s(u,u_k)u_k$ for some $s\in\N$, where $-\mcA u_k=\la_1u_k\ (1\leq k\leq s)$, which implies that
\begin{equation*}
-\mcA u=\sum_{k=1}^s (u,u_k)(-\mcA u_k)=\la_1u,
\end{equation*}
and hence \eqref{06.04.5} holds.

{\it Step 4: Show that if $0\neq u\in H_0^1(\Om;w)$, then
\begin{equation}\label{06.04.7}
\mbox{Either } u>0 \mbox{ in }\Om \mbox{ or } u<0 \mbox{ in }\Om. 
\end{equation}} 

Indeed, assume $\|u\|_{L^2(\Om)}=1$, and denote $u^+=\max\{u,0\}, u^-=\max\{-u,0\}$, and
\begin{equation*}
\al+\be=1,\quad \al =\int_\Om (u^+)^2\df x,\ \be=\int_\Om (u^-)^2\df x.
\end{equation*}
Note that $u^\pm\in H_0^1(\Om;w)$. Approximate $\Om$ by smooth domains $\Om_\e\ (\e>0)$ such that  $\Om\se \mcO(\Ga;2\e)\s \Om_\e\s \Om\se \mcO(\Ga;\e)$ for each $\e>0$ with $\mcO(\Ga;\e)=\{x\in\Om\colon \dist(\Ga, x)<\e\}$ and $\dist(\Ga,x)=\inf_{z\in \Ga}|z-x|$. Similar to the classical conclusion in $H_0^1(\Om_\e)=H_0^1(\Om;w)$, we have (see also Lemma 1.19 in \cite{Heinonen} at p.\! 19)
\begin{equation*}
\begin{split}
\nabla u^+=
\begin{cases}
\nabla u, &\mbox{a.e.\! on } \{u\geq 0\},\\
0, &\mbox{a.e.\! on } \{u\leq 0\},
\end{cases}\quad
\nabla u^-=
\begin{cases}
0, &\mbox{a.e.\! on } \{u\geq 0\},\\
-\nabla u, &\mbox{a.e.\! on }\{u\leq 0\},
\end{cases}
\end{split}
\end{equation*}
then $\int_\Om (\nabla u^+\cdot\nabla u^-)w\df x=0$. Hence from (i) we have
\begin{equation*}
\begin{split}
\la_1
&=\int_{\Om}(\nabla u\cdot\nabla u)w\df x=\int_\Om (\nabla u^+\cdot\nabla u^+)w\df x+\int_\Om (\nabla u^-\cdot\nabla u^-)w\df x\\
&\geq \la_1\|u^+\|_{L^2(\Om)}^2+\la_1\|u^-\|_{L^2(\Om)}^2=(\al+\be)\la_1=\la_1.
\end{split}
\end{equation*}
This implies that
\begin{equation*}
\int_\Om (\nabla u^+\cdot\nabla u^+)w\df x=\la_1\int_\Om (u^+)^2\df x,\quad \int_\Om (\nabla u^-\cdot\nabla u^-)w\df x=\la_1\int_\Om (u^-)^2\df x.
\end{equation*}
From {\it Step 3}, we have
\begin{equation*}
\begin{cases}
-\mcA u^+=\la_1u^+, &\mbox{in }\Om,\\
u^+=0, &\mbox{on }\pt \Om,
\end{cases} \mbox{ and }
\begin{cases}
-\mcA u^-=\la_1u^-, &\mbox{in }\Om,\\
u^-=0, &\mbox{on }\pt \Om
\end{cases}
\end{equation*}
in the weak sense.
We know that $u^+$ is H\"older continuous on $\Om$ from    \cite[Lemma 2.3.11, p.98]{Fabes} and $-\mcA u^+=\la_1u^+\geq 0$ in $\Om$. The function $u^+$ is therefore a supersolution. Assume there exists $\wh x_0\in\Om$ such that $u^+(\wh x_0)=0$, then there exists a smooth domain $\wh\Om\s \Om\se \mcO(\Ga;\e)$  with $\mcO(\Ga;\e)=\{x\in\Om\colon \dist(\Ga, x)<\e\}$ and $\dist(\Ga, x)=\inf_{z\in \Ga}|z-x|$, such that $x_0\in \wh\Om$. By the standard regularity of elliptic equations on $\wh\Om$ with $w\in C^\iy(\wh\Om)$, we have $u^+\in C^2(\Om)\cap C(\ol{\wh\Om})$. Note that $-\mcA$ is a uniformly elliptic equation on $\wh\Om$. By the strong maximum principle (\cite[Theorem 4, Chapter 6, p.33]{Evans}) on $\wh\Om$, we obtain $u^+=0$ on $\wh \Om$, hence $u^+=0$ on $\Om$. Hence $u^+>0$ in $\Om$ or else $u^+=0$ on $\Om$.
Similar arguments apply to $u^-$ to get \eqref{06.04.7}.
From the above, we obtain \eqref{06.04.7}.

{\it Step 5}. Finally, assume that $u_1, \wh u_1$ are two nontrivial weak solutions of \eqref{01.08.1} corresponding to $\la_1$. Assume $u_1>0$ and $\wh u_1>0$ on $\Om$ by {\it Step 4}, then $\int_\Om \wh u_1\df x\neq 0$, and hence there exists $0\neq a\in\R$ such that
\begin{equation*}
\int_\Om (u_1-a\wh u_1)\df x=0.
\end{equation*}
Note that $u_1-a\wh u_1$ is also a weak solution of \eqref{01.08.1}. From Step 4, we have $u_1=a\wh u_1$. Hence the eigenvalue $\la_1$ is simple.
\end{proof}
\begin{notation}\label{06.27.N1}
Here and in what follows, we denote $\Phi_i\ (i\in\N)$ as the $i$th orthonormal (i.e., $\|\Phi_{i}\|_{L^2(\Om)}=1$) eigenfunction of $\mcA$ corresponding to the $i$th eigenvalue $\la_i$:
\begin{equation*}
0<\la_1<\la_2\leq \la_3\leq \cdots.
\end{equation*}
And then  $\{\Phi_i\}_{i\in\N}$ forms an orthonormal basis of $L^2(\Om)$. We also use the notation $\la_i(M)$ and $\la_i(\Om)$ to distinguish the $i$th eigenvalues of $\mcA$ defined on $M$ and $\Om$, respectively. And we denote
\begin{equation*}
\la_1(M)=\inf_{0\neq u\in H_0^1(M;w)}\f{\int_\Om |\nabla u|^2w\df x}{\int_\Om u^2\df x}
\end{equation*}
for all bounded domains $M\s\R^2$ ($M$ may not be smooth).
\end{notation}
\begin{remark}\label{06.05.R1}
From Lemma \ref{06.04.L6}, we deduce that the second eigenfunction $\Phi_2$ of $\mcA$ corresponding to $\la_2$ must change sign.
\end{remark}
\begin{lemma}\label{06.28.L1}
Let $u=\sum_{i=1}^\iy u_i \Phi_i\in H_0^1(\Om;w)$ with $u_i=(u,\Phi_i)_{L^2(\Om)}$ for all $i\in\N$. We have $\nabla u=\sum_{i=1}^\iy u_i\nabla \Phi_i$, and
\begin{equation*}
u\in H^2(\Om;w)\Lra \sum_{i=1}^\iy u_i^2\la_i^2<\iy,
\end{equation*}
and
\begin{equation*}
-\mcA u=\sum_{i=1}^\iy u_i\la_i\Phi_i \mbox{ and } \|\mcA u\|_{L^2(\Om)}=\left(\sum_{i=1}^\iy u_i^2\la_i^2\right)^\f{1}{2}.
\end{equation*}
\end{lemma}
\begin{proof}
By the same argument as Step 1 in the proof of Lemma \ref{06.04.L6}, we have $\nabla u=\sum_{i=1}^\iy u_i\nabla\Phi_i$.
Let $u\in H^2(\Om;w)$. Take $\vp_n=\sum_{i=1}^n u_i\Phi_i\in H_0^1(\Om;w)\cap H^2(\Om;w)$ for each $n\in\N$, then from
\begin{equation*}
\begin{split}
(\mcA u, \mcA\vp_n)_{L^2(\Om)}
&=-\sum_{i=1}^n u_i\la_i \int_\Om (\mcA u)\Phi_i\df x=-\sum_{i=1}^n u_i\la_i\int_\Om u\mcA\Phi_i\df x\\
&=\sum_{i=1}^n u_i\la_i^2\int_\Om u\Phi_i\df x=\sum_{i=1}^n u_i^2\la_i^2
\end{split}
\end{equation*}
and $\|\mcA\vp_n\|_{L^2(\Om)}^2=\sum_{i=1}^n u_i^2\la_i^2$ we obtain
\begin{equation*}
\sum_{i=1}^nu_i^2\la_i^2\leq \|\mcA u\|_{L^2(\Om)}^2
\end{equation*}
for all $n\in\N$ by Cauchy inequality. This implies that $\sum_{i=1}^\iy u_i^2\la_i^2\leq \|\mcA u\|_{L^2(\Om)}^2<\iy$.
Let $\sum_{i=1}^\iy u_i^2\la_i^2<\iy$. For each  $\vp\in C_0^\iy(\Om)$,  we have
\begin{equation*}
\begin{split}
(-\mcA u, \vp)_{L^2(\Om)}
&=\int_\Om (\nabla u\cdot \nabla\vp)w\df x=\sum_{i=1}^\iy u_i \int_\Om (\nabla\Phi_i\cdot \nabla\vp)w\df x=\sum_{i=1}^\iy u_i\la_i\int_\Om \Phi_i\vp\df x.
\end{split}
\end{equation*}
Note that
\begin{equation*}
\int_\Om \left(\sum_{i=1}^n u_i\la_i\Phi_i\right)^2\df x=\sum_{i=1}^n u_i^2\la_i^2\leq \sum_{i=1}^\iy u_i^2\la_i^2<\iy \mbox{ for all } n\in\N,
\end{equation*}
i.e., $\sum_{i=1}^\iy u_i\la_i\Phi_i\in L^2(\Om)$, Hence
\begin{equation*}
(-\mcA u, \vp)_{L^2(\Om)}=\left(\sum_{i=1}^\iy u_i\la_i\Phi_i, \vp\right)_{L^2(\Om)}.
\end{equation*}
This implies that $\|\mcA u\|_{L^2(\Om)}\leq \sum_{i=1}^\iy u_i^2\la_i^2$ and $-\mcA u=\sum_{i=1}^\iy u_i\la_i\Phi_i$.
\end{proof}

 \subsection{Rayleigh Quotient}
For the uniformly elliptic operator, the Rayleigh quotient has been discussed in  \cite[Case 4, Section 1, Chapter VI, p.406]{Hilbert}
 or \cite{Sanayei}. For the degenerate elliptic operator, the Rayleigh quotient behaves similarly to that of the uniformly elliptic operator under normal circumstances.
\begin{theorem}\label{06.22.T2}
Let $\Phi_1,\cdots, \Phi_{i}\ (i\in\N)$ be the eigenfunctions of $\mcA$ corresponding to the eigenvalues $\la_1,\cdots, \la_{i}$. Define
\begin{equation*}
\mu=\inf_{f\in H_0^1(\Om;w)\atop f\bot \Span \{\Phi_1,\cdots, \Phi_{i}\}} \frac{\int_\Om |\nabla f|^2w\df x}{\int_\Om |f|^2\df x},
\end{equation*}
then $\mu=\la_{i+1}$, where $\Span\{\Phi_1,\cdots, \Phi_i\}$ denotes the linear span of $\Phi_1,\cdots, \Phi_i$.
\end{theorem}
\begin{proof}
Let $\{f_n\}_{n\in\N}\subset H_0^1(\Om;w)$ with $f_n\bot \Span\{\Phi_1,\cdots, \Phi_i\}$ for all $n\in\N$ satisfy
\begin{equation*}
\frac{\int_\Om |\nabla f_n|^2w\df x}{\int_\Om |f_n|^2\df x}\to \mu \mbox{ as } n\to\infty.
\end{equation*}
Assume $\|f_n\|_{L^2(\Om)}=1$ for all $n\in\N$. Then $\int_\Om |\nabla f_n|^2w\df x\to \mu$ as $n\to\infty$. Consequently, there exists a subsequence of $\{f_n\}_{n\in\N}$, still denoted by itself, and $f_0\in H_0^1(\Om;w)$ such that
\begin{equation}\label{06.27.2}
f_n\rightharpoonup f_0 \mbox{ weakly in } H_0^1(\Om;w).
\end{equation}
Therefore, we have
\begin{equation*}
f_n\to f_0 \mbox{ strongly in } L^2(\Om)
\end{equation*}
from (3) in Assumption \ref{06.27.A1}. This implies that $\|f_0\|_{L^2(\Om)}=1$, and  $f_0\bot \Span\{\Phi_1,\cdots, \Phi_i\}$.  From which and \eqref{06.27.2} we obtain
\begin{equation}\label{06.22.3}
\int_\Om |\nabla f_0|^2w\df x\leq \liminf_{n\to\infty} \int_\Om |\nabla f_n|^2w\df x=\mu.
\end{equation}
We shall prove $\mu=\la_{i+1}$ in the following.
Since $f_0\in L^2(\Om)$ and $\{\Phi_i\}_{i\in\N}$ is an orthonormal basis of $L^2(\Om)$, then we have
\begin{equation*}
f_0=\sum_{k=i+1}^\infty (f_0, \Phi_k)_{L^2(\Om)}\Phi_k \mbox{ and } \sum_{k=i+1}^\infty (f_0,\Phi_k)_{L^2(\Om)}^2=1
\end{equation*}
by $f_0\bot \Span \{\Phi_1,\cdots, \Phi_i\}$. This implies that
\begin{equation}\label{06.22.1}
\int_\Om |\nabla f_0|^2w\df x=\sum_{k=i+1}^\infty (f_0,\Phi_k)_{L^2(\Om)}^2\int_\Om |\nabla \Phi_k|^2w\df x=\sum_{k=i+1}^\infty\la_i (f_0,\Phi_i)_{L^2(\Om)}^2\geq \la_{i+1}.
\end{equation}
Now, we have
\begin{equation*}
\int_\Om |\nabla \Phi_{i+1}|^2w\df x=\la_{i+1} \mbox{ and } \Phi_{i+1}\bot \Span \{\Phi_1,\cdots, \Phi_i\},
\end{equation*}
then $\int_\Om |\nabla f_0|^2w\df x\leq \la_{i+1}$ from the definition of $\mu$, and hence
\begin{equation*}
\mu=\int_\Om |\nabla f_0|^2w\df x=\la_{i+1}
\end{equation*}
by \eqref{06.22.1}. This completes the proof of the theorem.
\end{proof}

Theorem \ref{06.22.T2} is equivalent to the following Theorem \ref{06.22.T3}.

\begin{theorem}\label{06.22.T3}
Let $i\in\N$. Then, 
\begin{equation*}
\la_{i}=\min_{H\in \La_i}\max_{f\in H\setminus\{0\}}\frac{\int_\Om|\nabla f|^2 w\df x}{\int_\Om |f|^2\df x},
\end{equation*}
where $\La_i$ is the set of subspaces of $H_0^1(\Om;w)$ of dimension $i$.
\end{theorem}
\begin{proof}
Denote
\begin{equation}\label{06.22.2}
\mu_i=\min_{H\in \La_i}\max_{f\in H\setminus\{0\}}\frac{\int_\Om|\nabla f|^2 w\df x}{\int_\Om |f|^2\df x}.
\end{equation}
We only need to show that $\la_i=\mu_i$ for all $i\in\N$.
Let $H_0=\Span\{\Phi_1,\cdots, \Phi_i\}$. Then  $H_0\in \La_i$ and
\begin{equation*}
\max_{f\in H_0\setminus\{0\}}\frac{\int_\Om |\nabla f|^2w\df x}{\int_\Om |f|^2\df x}=\la_i,
\end{equation*}
and hence $\mu_i\leq \la_i$.
Let $H\in \La_i$ and $H\neq H_0$.  Then there exists $h\in H$ such that $\|h\|_{L^2(\Om)}=1$ and $h\bot H_0$. Hence, 
\begin{equation*}
\int_\Om |\nabla h|^2w\df x\geq \la_{i+1}\geq \la_i
\end{equation*}
by Theorem \ref{06.22.T2}. This means that
\begin{equation*}
\mu_i=\max_{f\in H_0\setminus\{0\}}\frac{\int_\Om |\nabla f|^2w\df x}{\int_\Om |f|^2\df x}=\la_i.
\end{equation*}

The converse is obvious. 
This proves the theorem.
\end{proof}
\begin{remark}
The quotient $\frac{\int_\Om |\nabla f|^2w\df x}{\int_\Om |f|^2\df x}$ is also called {\it Rayleigh's quotient}. Hence, Theorems \ref{06.22.T2} and \ref{06.22.T3} associate  with  the Rayleigh quotient for the degenerate elliptic operator $\mcA$.
\end{remark}

\subsection{Courant's nodal domain theorem}
In this section, we provide a proof that Courant's nodal domain theorem remains valid in the setting of degenerate elliptic operators.

\begin{lemma}\label{06.23.L1}
Let $M\subset \Om$ be a smooth domain. Denote $\la_i(M)\ (i\in\N)$ and $\la_i(\Om)$ be the $i$-th eigenvalue of $\mcA$ with Dirichlet boundary condition on the domains $M$ and $\Om$ respectively. Then
\begin{equation*}
\la_i(\Om)\leq \la_i(M).
\end{equation*}
\end{lemma}
\begin{proof}
Let $i\in\N\cup\{0\}$. Let $\{\psi_j\}_{j\in \N}$ be an orthonormal basis of $L^2(M)$ with $\psi_j\ (j\in \N)$ the $j$-th eigenfunction of $\mcA$ on $M$ with respect to the eigenvalue $\la_j(M)$. Extend $\psi_j$ by
\begin{equation}\label{06.23.1}
E\psi_j=
\begin{cases}
\psi_j, &\mbox{in }M, \\
0, &\mbox{on }\Om\setminus M.
\end{cases}
\end{equation}
It is obvious that $E\psi_j\in H_0^1(\Om;w)\ (j\in\N)$.
By linear algebra, there exist $a_1,\cdots, a_{i+1}$, not all zero, satisfying
\begin{equation*}
\sum_{j=1}^{i+1}a_j(E\psi_j, \Phi_l)_{L^2(\Om)}=0 \mbox{ for all } l=1,\cdots, i.
\end{equation*}
Hence, the function
\begin{equation*}
f=\sum_{j=1}^{i+1}a_jE\psi_j
\end{equation*}
is orthogonal to $\Phi_1,\cdots, \Phi_i\in L^2(\Om)$. This implies that
\begin{equation*}
\la_{i+1}(\Om)\int_\Om|f|^2\df x\leq \int_\Om |\nabla f|^2w\df x=\int_M |\nabla f|^2w\df x\leq \la_{i+1}(M)\int_M|f|^2\df x
\end{equation*}
by Theorem \ref{06.22.T2}.
This shows that $\la_{i+1}(\Om)\leq \la_{i+1}(M)$.
\end{proof}
\begin{remark}
It is worth noting that $E\psi_j\in H_0^1(\Om;w)$ in  \eqref{06.23.1} is meaningful.  Indeed, there exists a sequence $\{\psi_j^n\}_{n\in\N}\subset  C_0^\infty(M)$ such that $\psi_j^n\to \psi_j$ strongly in $H_0^1(M;w)$, then $E\psi_j^n\to E\psi_j$ strongly in $H_0^1(\Om;w)$ also for $M\subset\Om$.
\end{remark}
\begin{definition}\label{06.23.D1}
The {\it nodal domain} of a continuous function $f: \Om\to \R$ is a connected component of the set
\begin{equation*}
\{f\neq 0\}\subset \Om.
\end{equation*}
\end{definition}
\begin{remark}
From \cite{Fabes}, we know that $\Phi_i\ (i\in \N)$ is H\"older continuous on $\Om$, hence it is meaningful to consider the nodal domain associated with  $\Phi_i$. We note that the nodal domain of $f$ may not be a smooth domain in general.
\end{remark}
\begin{lemma}\label{06.23.L1}
Let $\Phi_i$ be the $i$th eigenfunction of $\mcA$ with respect to the eigenvalue $\la_i$, and let $\Om_j$ be a nodal domain of $\Phi_i$. Then $\Phi_i\in H_0^1(\Om_j)$ and
\begin{equation*}
\la_i=\la_1(\Om_j).
\end{equation*}
\end{lemma}
\begin{proof}
Assume that $\Phi_i>0$ on $\Om_j$. For every  $\e>0$, set
\begin{equation*}
\Phi_i^{j}=
\begin{cases}
\Phi_i(x), &x\in \Om_j,\\
0, &x\in \Om\setminus \Om_j,
\end{cases}
\end{equation*}
and
\begin{equation*}
\Phi_{i,\e}^{j}=
\begin{cases}
\Phi_i^{j}-\e, &\mbox{on }\Om_{j,\e},\\
0, &\mbox{on }\Om\setminus \Om_{j,\e},
\end{cases},\quad \Om_{j,\e}=\left\{x\in \Om_j\colon \Phi_i^{j}(x)>\e\right\}.
\end{equation*}
By the standard elliptic regularity and $\Gamma\subset \partial\Om$ and $w\in C(\overline{\Om})\cap C^\infty(\Om)$, we have $\Phi_{i}\in C^\infty(\overline{\Om_{j,\e}})$ and $\Phi_i\in C^\infty(\Om)$. However, $\Phi_i$ may not possess classical first-order partial derivatives on $\Gamma\subset\overline{\Om}$, which implies $\Phi_i\notin C^\infty(\overline{\Om})$ in general. And by Sard's theorem \cite{Sard}, there exists regular values $\{\e_l\}_{l\in\N}$, such that $\e_l\to 0$ as $l\to\infty$ and we have  $\partial\Om_{j,\e_l}$ is smooth (i.e., the critical points are not on $\partial\Om_{j,\e_l}$) and $\Phi_{i,\e_l}^j\in H_0^1(\Om_{i,\e_l};w)=H_0^1(\Om_{i,\e_l})$. We simplify denote $(l\in\N)$
\begin{equation*}
\Om_{j,l}=\Om_{j,\e_l},\quad \Phi_{k,l}^{j}=\Phi_{k,\e_l}^j.
\end{equation*}
Then,
\begin{equation*}
\begin{split}
\la_i(\Om)\int_{\Om_{j}}\Phi_{i,l}^j\Phi_i^j\df x
&=\la_i(\Om)\int_{\Om}\Phi_{i,l}^j\Phi_i\df x=\int_\Om (\nabla \Phi_{i,l}^j\cdot\nabla \Phi_i)w\df x\\
&=\int_{\Om_{j}}|\nabla \Phi_{i,l}^j|^2w\df x\geq \la_1(\Om_j)\int_{\Om_{j}}|\Phi_{i,l}^j|^2\df x
\end{split}
\end{equation*}
by using $\Phi_{i,l}^j\in H_0^1(\Om;w)$ as a test function.  Letting $l\to\infty$, by $\Phi_{i,l}^j\to \Phi_i^j$ strongly in $L^2(\Om_j)$,  we obtain
\begin{equation*}
\begin{split}
\la_i(\Om)\int_{\Om_j}|\Phi_i^j|^2\df x\geq \la_1(\Om_j)\int_{\Om_j}|\Phi_i^j|^2\df x.
\end{split}
\end{equation*}
This implies that $\la_i(\Om)\geq \la_1(\Om_j)$.
We now proceed to prove the opposite inequality.
Let $\e>0$ and $\e$ is a regular value of $\Phi_i^j$. Choose $\vp_\e>0$ be the first eigenfunction of $\mcA$ on $\Om_{j,\e}$ with respect to the eigenvalue $\la_1(\Om_{j,\e})$. We obtain
\begin{equation*}
\begin{split}
\la_i(\Om)\int_{\Om_{j,\e}}\vp_\e \Phi_i^j\df x
&=-\int_{\Om_{j,\e}}\vp_\e\Div(w\nabla \Phi_i^j)\df x=-\int_{\Om_{i,\e}}\Div(w\nabla\vp_\e)\Phi_i^j\df x+\int_{\partial \Om_{j,\e}}\Phi_i^j \frac{\partial\vp_\e}{\partial \nu_w}\df S\\
&\leq -\int_{\Om_{i,\e}}\Div(w\nabla\vp_\e)\Phi_i^j\df x=\la_1(\Om_{j,\e})\int_{\Om_{j,\e}}\vp_\e \Phi_i^j\df x
\end{split}
\end{equation*}
since $\Om_{j,\e}$ is smooth and $\varphi_\varepsilon,\Phi_i^j\in H^2(\Om_{j,\e})$, where $\frac{\partial\vp_\e}{\partial \nu_w}=(\nabla\vp_\e\cdot\nu)w\leq 0$. This implies that
\begin{equation}\label{06.23.2}
\la_i(\Om)\leq \la_1(\Om_{j,\e})
\end{equation}
for all regular values $\e>0$ of $\Phi_i^j$, since $\vp_\e>0$ and $\Phi_i^j>0$ on $\Om_{j,\e}$.
Finally, we prove: for all $\varepsilon > 0$ such that $\varepsilon$ is a regular value of $\Phi_i^j$, we have
\begin{equation*}
\lim_{\e\to 0}\la_1(\Om_{j,\e})=\la_1(\Om_j).
\end{equation*}
And this combine \eqref{06.23.2} and Sard's theorem we obtain $\la_i(\Om)\leq \la_1(\Om_j)$.
For each $\de>0$, by the definition of  $\la_1(\Om_j)$, i.e.,
\begin{equation*}
\la_1(\Om_j)=\inf_{0\neq f\in H_0^1(\Om_j;w)}\frac{\int_\Om |\nabla f|^2w\df x}{\int_\Om |f|^2\df x},
\end{equation*}
there exists $f_\de\in C_0^\infty(\Om_j)$ such that
\begin{equation*}
\frac{\int_\Om |\nabla f_\de|^2w\df x}{\int_\Om |f_\de|^2\df x}\leq \la_1(\Om_j)+\de.
\end{equation*}
Note that there exists $\e_0>0$ such that $\supp f_\de\subset \Om_{j,\e}$ for all regular values $\varepsilon  \in (0, \varepsilon _0]$ of $\Phi_i^j$, which implies that
\begin{equation*}
\la_1(\Om_{j,\e})\leq \frac{\int_\Om |\nabla f_\de|^2w\df x}{\int_\Om |f_\de|^2\df x}.
\end{equation*}
Hence, for each $\de>0$, there exists $\e_0\geq 0$,  such that for all regular values $\varepsilon  \in (0, \varepsilon _0]$ of $\Phi_i^j$, we have
\begin{equation*}
\la_1(\Om_j)\leq \la_1(\Om_{j,\e})\leq \la_1(\Om_j)+\de.
\end{equation*}
Therefore, $\lim_{\e\to 0}\la_1(\Om_{j,\e})=\la_1(\Om_j)$.
\end{proof}
\begin{remark}
In contrast to the uniformly elliptic case, it is difficult to show that $\Phi_i \in C^\infty(\overline{\Omega})$ for degenerate elliptic equations.
\end{remark}
From Lemma \ref{06.23.L1} we obtain the following Theorem \ref{06.22.T1}. 

\begin{theorem}[Courant nodal domain theorem]\label{06.22.T1}
The eigenfunction $\Phi_i\ (i\in\N)$ has at most $i$ nodal domains.
\end{theorem}
\begin{proof}
Suppose that $\Phi_i\ (i\in\N)$ has at least $i+1$ nodal domains, denote these nodal domains by $\Om_1,\cdots,\Om_{i+1}$. From Lemma \ref{06.23.L1}, we know that $\la_i(\Om)=\la_1(\Om_j)$, where $\Om_j$ is a nodal domain of $\Phi_i$. For each nodal domain $\Om_j$ of $\Phi_i$, consider the first eigenfunction $f_j$ in $H_0^1(\Om_j;w)$, we know by definition of eigenfunction that
\begin{equation*}
\int_{\Om_j} |\nabla f_j|^2w\df x=\la_i(\Om)\int_{\Om_j}|f_j|^2\df x.
\end{equation*}
Consider the functions $\{f_j\}_{j=1}^{i+1}$, by linear algebra, there exist $a_1,\cdots, a_{i}$, not all zero, such that  the function
\begin{equation*}
g=\sum_{j=1}^{i} a_jf_j+0\cdot f_{i+1}
\end{equation*}
satisfies $g\bot\Phi_j, j=1,\cdots, i-1$, where $a_j, j=1,\cdots, i$ satisfy the following equations
\begin{equation*}
\sum_{j=1}^{i}a_j(f_j,\Phi_l)_{L^2(\Om)}=0 \mbox{ for all } l=1,\cdots, i-1.
\end{equation*}
Note that $g\in H_0^1(\Om;w)$, and from $C_0^\iy(\Om_j)$ is dense in $H_0^1(\Om_j;w)\ (j=1,\cdots, i+1)$,  we have
\begin{equation*}
\begin{split}
\int_{\Om}|\nabla g|^2w\df x
&=\sum_{j=1}^{i} a_j^2\int_{\Om_j}|\nabla f_j|^2w\df x=\sum_{j=1}^{i} \la_i(\Om)a_j^2\int_{\Om_j}|f_j|^2\df x=\la_i(\Om)\int_\Om |g|^2\df x,
\end{split}
\end{equation*}
hence,  by the same argument as {\it Step 3}  in the proof of Lemma \ref{06.04.L6}, we have  $g$ is an eigenfunction of $\mcA$ with respect to the eigenvalue $\la_i(\Om)$. Note that $g=0$ on $\Om_{i+1}$, it is a contradiction by the unique continuation property (see Remark \ref{07.16.R1}) of $\Phi_i$ (i.e., if $\Phi_i=0$ on a nonempty open subset $\Om_{i+1}\subset\Om$, then $\Phi_i=0$ on $\Om$). This proves the theorem.
\end{proof}

\begin{remark}\label{07.16.R1}
In the proof of Theorem \ref{06.22.T1}, we must use the unique continuation property (i.e., if $\Phi_i=0$ on a nonempty open subset $\om\subset \Om$, then $\Phi_i=0$ on $\Om$) for the eigenfunction $\Phi_i$, this is \cite[Theorem 5.2, Chapter 5.2, p. 186]{Rousseau}.
Indeed, for any $\e>0$, take a smooth domain $\Om\subset \mcO(\Ga;2\e)\subset  \Om_\e\subset \Om\subset \mcO(\Ga;\e)$ with $\mcO(\Ga;\e)=\{x\in\Om\colon \dist(\Ga, x)<\e\}$ and $\dist(\Ga,x)=\inf_{z\in \Ga}|z-x|$, then $\mcA$ is a uniformly elliptic equation on $\Om_\e$. Now, we have
\begin{equation*}
\left|(\mcA+\la)\Phi_i(x)\right|\leq 2\la_i|\Phi_i(x)| \mbox{ a.e.\! in }\Om_\e,
\end{equation*}
and $\Phi_i=0$ on $\om$, apply \cite[Theorem 5.2, Chapter 5.2]{Rousseau}, we get $\Phi_i=0$ on $\Om_\e$. Letting $\e\to 0$, we obtain $\Phi_i=0$ on $\Om$. We note that the unique continuation property as above is a qualitative, rather than a quantitative, property.
\end{remark}
From Lemma \ref{06.04.L6} and Theorem \ref{06.22.T1} we obtain the following Corollary \ref{06.05.C1}. 

\begin{corollary}\label{06.05.C1}
The eigenfunction $\Phi_2$  has exactly two nodal domains.
\end{corollary} 
\begin{remark}
Provided that the $A_2$ weight function $w$ is degenerate only at discrete points in the interior of $\Omega$ and Assumption \ref{06.27.A1} is satisfied, the same conclusions continue to hold. But we do not know what will happen when $w\geq 0$ a.e. in $\Om$, since Remark \ref{07.16.R1} may not hold anymore.
\end{remark}

\section{Simple  eigenfunctions}\label{S4}

In this section, we shall prove the fact that $\{\rho \in L^\infty(\Omega) \colon \mathcal{A} + \rho \text{ has simple eigenvalues}\}$
encompasses a residual subset within $(L^\infty(\Omega), |\cdot|_\infty)$. To establish this result, we will introduce an analytic perturbation for $\mcA$. As noted in what follows, if any difficulties arise, we may assume $\rho \in L^\infty(\Omega)$ with $\rho \geq 1$. The equivalent norm on $H_0^1(\Omega;w)$ is defined as
\begin{equation*}
\|u\|_{H^1(\Omega;w)}=\left(\int_\Omega |\nabla u|^2w\,dx+\int_\Omega \rho u^2\,dx\right)^{\frac{1}{2}},
\end{equation*}
and its inner product is given by
\begin{equation*}
(u,v)_{H^1(\Omega;w)}=\int_\Omega (\nabla u\cdot\nabla v)w\,dx+\int_\Omega \rho uv\,dx,
\end{equation*}
throughout this section.
\begin{definition}\label{06.24.D1}
Let $(E,\|\cdot\|_E)$ be a normed space. Let $T>0$ and let
\begin{equation*}
f: (-T,T)\rightarrow E
\end{equation*}
be a function. We say that $f$ is {\it analytic} in $E$ for $|\tau|<T$ if there exists a sequence $\{f_j\}_{j=0}^\infty\subseteq E$ such that
\begin{equation*}
\lim_{m\rightarrow\infty}\left\|f(\tau)-\sum_{j=0}^mf_j\tau^j\right\|_E=0 \quad \text{for } |\tau|<T.
\end{equation*}
We say that $f$ is {\it absolutely analytic} for $|\tau|<T$ if, in addition, the real series
\begin{equation*}
\sum_{j=0}^\infty \|f_j\|_E|\tau|^j
\end{equation*}
converges for $|\tau|<T$.
\end{definition}
If $E, F$ are Banach spaces, let $\mcL(E; F)$ denote the space of bounded linear operators from $E$ to $F$; $\mcL(E;F)$ is a Banach space with the operator norm.
The following Lemmas \ref{06.25.L1} and \ref{06.25.L2} are   \cite[Lemmas 1 and 2, Appendix II, p. 366]{Albert} or \cite[Section 2, Chapter VII, p.366]{Kato}.

\begin{lemma}\label{06.25.L1}
Suppose that in a neighborhood of $\tau=0$, $G(\tau)$ is an absolutely analytic function in $\mcL(E; F)$ and $v(\tau)$ is analytic in $E$.
 Then,  $G(\tau)v(\tau)$ is analytic in $F$ for small $\tau$.
\end{lemma}
\begin{lemma}\label{06.25.L2}
If $G(\tau)$ is absolutely analytic in $\mcL(E; F)$ for small $\tau$ and $G(0)$ is invertible, then for $\tau$ in a small neighborhood of $0$, $G(\tau)$ is invertible and $[G(\tau)]^{-1}$ is absolutely analytic in $\mcL(F; E)$ for small $\tau$.
\end{lemma}

The following Lemma \ref{06.25.L3} is abstract theorems of \cite[Appendix II, p.92-93]{Albert}. It is merely a conclusion in functional analysis.
\begin{lemma}\label{06.25.L3}
(Hypotheses): Let $E$ be a real Hilbert space with inner product $\langle\cdot, \cdot\rangle$. Suppose that  $P$ is a symmetric, densely defined operator on $E$ and assume $\lambda$ is an eigenvalue of $P$ with multiplicity $h$. Assume there is a bounded linear operator $Q: E\rightarrow E$ such that $Q\pi_\lambda=0$ and
\begin{equation*}
Q(P+\lambda\, \text{id})=\text{id}-\pi_\lambda,
\end{equation*}
where $\pi_\lambda: E\rightarrow \ker(P+\lambda)$ is the orthogonal projection. Let  $R(\tau)$ be an absolutely analytic function in $\mcL(E; F)$ for small $\tau$, self-adjoint for every  $\tau$ and such that $R(0)=0$. Let $P(\tau)=P+R(\tau)$. Then, there exist $h$ functions $\lambda^{(1)}(\tau),\cdots, \lambda^{(h)}(\tau)$ analytic in $\R$ for small $\tau$ and $h$ functions $u^{(1)}(\tau),\cdots, u^{(h)}(\tau)$ analytic in $E$ for small $\tau$, such that for $j=1,\cdots, h$, 

\begin{itemize} 

\item  $\lambda^{(j)}(0)=\lambda_j$;

\item  For every  $\tau$, $u^{(j)}(\tau)$ is an eigenvector of $P(\tau)$ with eigenvalue $\lambda^{(j)}(\tau)$;

\item For every $\tau$, $\{u^{(1)}(\tau), \cdots, u^{(h)}(\tau)\}$ is an orthonormal set in $E$.
\end{itemize} 
\end{lemma}

For the degenerate partial differential operator $\mcA$, we can obtain the following Corollary \ref{06.25.C1}.

\begin{corollary}\label{06.25.C1}
Let $E=L^2(\Omega), P=\mcA$ in Lemma \ref{06.25.L3}. Then the degenerate partial differential operator
\begin{equation*}
\mcA: H_0^1(\Omega;w)\cap H^2(\Omega;w)\rightarrow L^2(\Omega)
\end{equation*}
defined in \eqref{06.25.1} satisfies the hypotheses in Lemma \ref{06.25.L3}.
\end{corollary}
\begin{proof}
Let $u\in H_0^1(\Omega;w)\cap H^2(\Omega;w)$ satisfy $\pi_\lambda u=0$. Then $u=\sum_{i=1}^\infty u_i\Phi_i$ with $u_i\in \R, i\in\N$, and from Lemma \ref{06.28.L1} we have
\begin{equation*}
0=(\mcA+\lambda)u=-\sum_{i=1}^\infty u_i\lambda_i\Phi_i+\lambda\sum_{i=1}^\infty u_i\Phi_i=-\sum_{\lambda_i\neq \lambda}(\lambda_i-\lambda)u_i\Phi_i.
\end{equation*}
This implies that $u_i=0$ for all $\lambda_i\neq \lambda$. Hence
\begin{equation*}
\pi_\lambda (u)=\sum_{\lambda_i=\lambda} u_i\Phi_i.
\end{equation*}
From above, for every  $f\in L^2(\Omega)$, since $f=\sum_{i=1}^\infty f_i\Phi_i$, we define
\begin{equation}\label{06.25.2}
\pi_\lambda (f)=\sum_{\lambda_i=\lambda} f_i\Phi_i.
\end{equation}
It is clear that $\pi_\lambda (f)\in \ker(\mcA+\lambda)$ since the multiplicity of the eigenspace of $\mcA$ with eigenvalue $\lambda$ is finite.
Next, for every  $f\in L^2(\Omega)$, we have $f=\sum_{i=1}^\infty f_i\Phi_i$ with $f_i\in \R, i\in \N$. Define
\begin{equation}\label{06.25.4}
Q(f)=Q\left(\sum_{i=1}^\infty f_i\Phi_i\right)=-\sum_{\lambda_i\neq \lambda} \frac{1}{\lambda_i-\lambda}f_i\Phi_i.
\end{equation}
Then from \eqref{06.25.2} we have
\begin{equation*}
Q\pi_\lambda (f)=0,
\end{equation*}
and
\begin{equation*}
\begin{split}
\|Q(f)\|_{L^2(\Omega)}^2
&=\sum_{\lambda_i\neq\lambda}\frac{1}{(\lambda_i-\lambda)^2}f_i^2\leq \frac{1}{\min_{\lambda_i\neq \lambda}|\lambda_i-\lambda|^2}\sum_{i=1}^\infty f_i^2=\frac{1}{\min_{\lambda_i\neq \lambda}|\lambda_i-\lambda|^2}\|f\|_{L^2(\Omega)}.
\end{split}
\end{equation*}
This implies that $Q: L^2(\Omega)\rightarrow L^2(\Omega)$ is a bounded linear operator.
Finally, for every $u\in H_0^1(\Omega;w)\cap H^2(\Omega;w)$, we have $u=\sum_{i=1}^\infty u_i\Phi_i$ with $u_i\in \R, i\in\N$, and
by \eqref{06.25.2} and \eqref{06.25.4} and Lemma \ref{06.28.L1}, 
\begin{equation*}
\begin{split}
Q(\mcA+\lambda\,\text{id})u
&=Q\left(-\sum_{i=1}^\infty \lambda_iu_i\Phi_i+\lambda \sum_{i=1}^\infty u_i\Phi_i\right)\\
&=-Q\left( \sum_{\lambda_i\neq\lambda}(\lambda_i-\lambda)u_i\Phi_i\right)
=\sum_{\lambda_i\neq\lambda}u_i\Phi_i=\left(\text{id}-\pi_\lambda\right)u. 
\end{split}
\end{equation*}
This complete the proof of the corollary. 
\end{proof}

\begin{remark}\label{06.25.R1}
Let $\rho\in L^\infty(\Omega)$.  Consider the eigenvalue problem
\begin{equation}\label{06.25.3}
\mcA_\rho u:=\mcA u+\rho u=-\lambda u \quad \text{for } u\in H_0^1(\Omega;w)\cap H^2(\Omega;w).
\end{equation}
From the classical Fredholm alternative law, we have:
either
\begin{equation*}
\begin{cases}
\mcA_\rho u+\lambda u=0, &\text{in }\Omega, \\
u=0, &\text{on }\partial\Omega
\end{cases}
\end{equation*}
has nonzero solutions and the multiplicity of $\lambda$ is finite, or
\begin{equation*}
\begin{cases}
\mcA_\rho u+\lambda u=h, &\text{in }\Omega,\\
u=0, &\text{on }\partial\Omega
\end{cases}
\end{equation*}
has a unique solution for $h\in L^2(\Omega)$ and $\lambda\in\R$ is not an eigenvalue of $\mcA_\rho$.
Replace $\mcA$ by $\mcA_\rho$, we can assume $\rho\geq 1$. The reason is as follows: if not, we replace $\mcA_\rho$ by $\mcA_\rho-\gamma$ with $\gamma=|\rho|_\iy+1$, then $(\mcA_\rho-\gamma)u+(\lambda+\gamma)u=0$, which implies that $\lambda+\gamma$ is the eigenvalue of $\mcA_\rho-\gamma$ and $u$ is the eigenfunction of $\mcA_\rho-\gamma$ with respect to $\lambda+\gamma>0$. Define $\pi_\lambda$ as \eqref{06.25.2}, i.e.,
\begin{equation}\label{06.25.6}
\pi_\lambda: L^2(\Omega)\rightarrow \ker(\mcA_\rho+\lambda), \ f=\sum_{i=1}^\infty f_i\Phi_i\mapsto \sum_{\lambda_i=\lambda} f_i\Phi_i,
\end{equation}
and define $Q_\rho$ as \eqref{06.25.4}, i.e.,
\begin{equation}\label{06.25.7}
Q_\rho(f)=-\sum_{\lambda_i\neq\lambda} \frac{1}{\lambda_i-\lambda} f_i\Phi_i \quad \text{for } f=\sum_{i=1}^\infty f_i\Phi_i,
\end{equation}
then $Q_\rho$ satisfies $Q_\rho\pi_\lambda=0$ is a bounded linear operator $Q_\rho: E\rightarrow E$ and
\begin{equation}\label{06.25.5}
Q_\rho(\mcA_\rho+\lambda\,\text{id})=\text{id}-\pi_\lambda.
\end{equation}
We call $Q_\rho$ the pseudo-differential inverse of $\mcA_\rho+\lambda$. Moreover, since $\{\lambda_i^{-\frac{1}{2}}(\nabla\Phi_i+\rho^\frac{1}{2}\Phi_i)\}_{i\in\N}$ is an orthonormal basis of $H_0^1(\Omega;w)$ with equivalent norm
\begin{equation*}
\|u\|_{H^1(\Omega;w)}^2=\int_\Omega |\nabla u|^2w\,dx+\int_\Omega \rho u^2\,dx
\end{equation*}
(see Step 1 in the proof of Lemma \ref{06.04.L6}),  we have
\begin{equation*}
\begin{split}
\|Q_\rho(f)\|_{H^1(\Omega;w)}^2
&=\sum_{\lambda_i\neq\lambda}\frac{1}{(\lambda_i-\lambda)^2} f_i^2\left(\int_\Omega |\nabla \Phi_i|^2w\,dx+\int_\Omega \rho\Phi_i^2\,dx\right)\\
&=\sum_{\lambda_i\neq \lambda}\frac{\lambda_i}{(\lambda_i-\lambda)^2}f_i^2\leq \max_{\lambda_i\neq\lambda}\frac{\lambda_i}{(\lambda_i-\lambda)^2}\|f\|_{L^2(\Omega)}^2,
\end{split}
\end{equation*}
and from Lemma \ref{06.28.L1} we have
\begin{equation*}
\begin{split}
\|Q_\rho(f)\|_{H^2(\Omega;w)}^2&=\sum_{\lambda_i\neq \lambda}\frac{\lambda_i^2}{(\lambda_i-\lambda)^2}f_i^2+\sum_{\lambda_i\neq\lambda} \frac{\lambda_i}{(\lambda_i-\lambda)^2}f_i^2\\
&\leq \left[\max_{\lambda_i\neq\lambda}\frac{\lambda_i^2}{(\lambda_i-\lambda)^2}+\max_{\lambda_i\neq\lambda}\frac{\lambda_i}{(\lambda_i-\lambda)^2}\right]\|f\|_{L^2(\Omega)}^2,
\end{split}
\end{equation*}
and hence
\begin{equation}\label{06.26.1}
Q_\rho: L^2(\Omega)\rightarrow H_0^1(\Omega;w)\cap H^2(\Omega;w)
\end{equation}
is bounded linear operator from \eqref{06.04.2} (i.e., $\lambda_k\rightarrow \infty$ as $k\rightarrow\infty$).
\end{remark}

The following Lemma \ref{06.25.L4}  is the \cite[Proposition, Appendix II, p.94]{Albert}.

\begin{lemma}\label{06.25.L4}
Under the conditions of  Lemma \ref{06.25.L3}, there exists an analytic function $\lambda(\tau)$ in $\R$ for small $\tau$ and an analytic function $u(\tau)$ in $E$ for small $\tau$ such that

(i) $\lambda(0)=\lambda$;

(ii) $u(\tau)$ is an eigenvector of $P(\tau)$ with eigenvalue $\lambda(\tau)$.
\end{lemma}

The following Lemma \ref{06.26.L1} is just  \cite[Theorem 5.1, p.50]{Albert}.
\begin{lemma}\label{06.26.L1}
Let $\rho_1,\rho_2\in (L^\infty(\Omega), |\cdot|_\infty)$ be two potential functions, $\lambda_n^1$ and $\lambda_n^2$ be the $n$th eigenvalues of the operators $\mcA_{\rho_1}$ and $\mcA_{\rho_2}$ respectively. Then,
\begin{equation*}
|\lambda_n^1-\lambda_n^2|\leq |\rho_1-\rho_2|_\infty.
\end{equation*}
\end{lemma}
\begin{proof}
Let $\Phi_1^1,\cdots, \Phi_{n}^1$ be the orthonormal basis of $\mcA_{\rho_1}$ with respect to $\lambda_1^1,\cdots, \lambda_{n}^1$, where
\begin{equation*}
\lambda_1^1\leq \lambda_2^1\leq \lambda_3^1\leq \cdots.
\end{equation*}
Denote $H_0=\Span\{\Phi_1^1,\cdots, \Phi_n^1\}$.
Then,
\begin{equation*}
\begin{split}
\lambda_n^2-\lambda_n^1
&\leq \max_{f\in H_0\setminus\{0\}}\frac{\int_\Omega |\nabla f|^2w\,dx+\int_\Omega \rho_2 |f|^2\,dx}{\int_\Omega |f|^2\,dx}-\lambda_n^1\\
&\leq \max_{f\in H_0\setminus\{0\}}\left[\frac{\int_\Omega |\nabla f|^2w\,dx+\int_\Omega \rho_1 |f|^2\,dx}{\int_\Omega |f|^2\,dx}+\frac{\int_\Omega (\rho_2-\rho_1)|f|^2\,dx}{\int_\Omega |f|^2\,dx}\right]-\lambda_n^1\\
&\leq \max_{f\in H_0\setminus\{0\}}\left[\frac{\int_\Omega |\nabla f|^2w\,dx+\int_\Omega \rho_1 |f|^2\,dx}{\int_\Omega |f|^2\,dx}\right]-\lambda_n^1+\max_{f\in H_0\setminus\{0\}}\frac{\int_\Omega (\rho_2-\rho_1)|f|^2\,dx}{\int_\Omega |f|^2\,dx}\\
&\leq |\rho_2-\rho_1|_\infty
\end{split}
\end{equation*}
by Theorem \ref{06.22.T3}. By the same argument as above, we also have $\lambda_n^1-\lambda_n^2\leq |\rho_1-\rho_2|_\infty$. These prove the lemma.
\end{proof}

The following Lemma \ref{06.24.L1}  is   \cite[Theorem 4.5, p.46]{Albert}. It depends on the abstract theorem in Appendix II Perturbation Theory in \cite{Albert} from pages  92-93 (i.e., Lemma \ref{06.25.L3}). The operator $\mcA$ in this paper  meets the conditions of this  abstract theorem,  these are all functional analysis.
\begin{lemma}[Perturbation]\label{06.24.L1}
Let $\lambda$ be an eigenvalue of multiplicity $h$ of $\mcA+\rho$. Suppose $\rho(\tau)$ is absolutely analytic in $(L^\infty(\Omega), |\cdot|_\infty)$ on some neighborhood of $\tau=0$. Then there exists a number $T_0>0$, $h$ functions $\lambda^{(1)}(\tau),\cdots, \lambda^{(h)}(\tau)$ analytic in $R=\{\tau\in\R\colon |\tau|<T_0\}$ and $h$ functions $u^{(1)}(\tau),\cdots, u^{(h)}(\tau)$ analytic in $H_0^1(\Omega;w)\cap H^2(\Omega;w)$ for $|\tau|\leq T_0$, such that

\begin{itemize} 

\item $\lambda^{(j)}(0)=\lambda, j=1,\cdots, h$;

\item $u^{(j)}(\tau)$ is an eigenfunction of $\mcA+\rho(\tau)$, with eigenvalue $\lambda^{(j)}(\tau)$ for $|\tau|\leq T_0, j=1,\cdots, h$;

\item $\{u^{(1)}(\tau),\cdots, u^{(h)}(\tau)\}$ is an orthonormal set in $L^2(\Omega)$ for each $\tau, |\tau|\leq T_0$;

\item  For every open interval $I\subseteq R$ such that $\overline{I}$ contains only the eigenvalue $\lambda$ of $\mcA+\rho$, there is a $T>0$ such that for $|\tau|<T$ there are exactly $h$ eigenvalues (counting multiplicity)

\end{itemize} 
\begin{equation*}
\lambda^{(1)}(\tau),\cdots, \lambda^{(h)}(\tau) \text{ of } \mcA+\rho(\tau) \text{ in } I.
\end{equation*}
\end{lemma}
 \begin{proof}
Let $E = L^2(\Omega)$ and define $\mcA_\rho=\mcA+\rho$, where $\rho\in L^\infty(\Omega)$. Denote $Q_\rho$ as the pseudo-differential inverse of $\mcA_\rho+\lambda$, as given by \eqref{06.25.7}. Suppose $\rho(\tau)$ is an absolutely analytic function in the Banach space $(L^\infty(\Omega), |\cdot|_\infty)$ with $\rho(0)=\rho$, and let $R(\tau)=\rho(\tau)-\rho$. The first three assertions are established by Lemma \ref{06.25.L3}.
From the second assertion, we have $[\mcA_\rho(\tau)+\lambda(\tau)]u(\tau)=0$. Applying Lemma \ref{06.25.L4}, it follows that for small $\tau$, $u(\tau)$ is analytic in $L^2(\Omega)$ and $\lambda(\tau)$ is analytic in $\R$. Moreover, we obtain the equation
\begin{equation*}
\left[\mcA_\rho+\lambda\right]u(\tau)=-\left[R(\tau)+\lambda(\tau)-\lambda\right]u(\tau).
\end{equation*}
Rearranging this equation yields
\begin{equation*}
\begin{split}
u(\tau)=-Q_\rho\big(\left[R(\tau)+\lambda(\tau)-\lambda\right]u(\tau)\big)+v(\tau),
\end{split}
\end{equation*}
where $v(\tau)=\pi_\lambda u(\tau)\in \ker(\mcA_\rho+\lambda)$. Consequently, we can express $u(\tau)$ as
\begin{equation*}
u(\tau)=-Q_\rho\left[R(\tau)+\lambda(\tau)-\lambda\right]u(\tau)+v(\tau), \text{ with } v(\tau)=\sum_{j = 1}^h c_j(\tau)u_j,
\end{equation*}
where $\{u_1,\cdots, u_h\}$ forms an orthonormal basis of $\ker(\mcA_\rho+\lambda)$. Given that $\ker(\mcA_\rho+\lambda)\subseteq H_0^1(\Omega;w)\cap H^2(\Omega;w)$ and by Lemma \ref{06.25.L1}, we find that
\begin{equation*}
c_j(\tau)=(u(\tau), u_j)_{L^2(\Omega)}+\big(Q_\rho\left[R(\tau)+\lambda(\tau)-\lambda\right]u(\tau), u_j\big)_{L^2(\Omega)}, \quad j = 1,\cdots, h
\end{equation*}
are analytic functions of $\tau$ in $\R$. Thus, $v(\tau)=\sum_{j = 1}^h c_j(\tau)u_j$ is analytic in $H_0^1(\Omega;w)\cap H^2(\Omega;w)$. Since $R(\tau)+\lambda(\tau)-\lambda$ is analytic in $L^\infty(\Omega)$, it follows from Lemma \ref{06.25.L1} that $[R(\tau)+\lambda(\tau)-\lambda]u(\tau)$ is analytic in $L^2(\Omega)$. Noting that
\begin{equation*}
Q_\rho: L^2(\Omega)\rightarrow H_0^1(\Omega;w)\cap H^2(\Omega;w)
\end{equation*}
is a bounded operator as shown in \eqref{06.26.1}, we conclude by Lemma \ref{06.25.L1} that $Q_\rho[R(\tau)+\lambda(\tau)-\lambda]u(\tau)$ is analytic in $H_0^1(\Omega;w)\cap H^2(\Omega;w)$. These results collectively imply that $u(\tau)$ is analytic in $H_0^1(\Omega;w)\cap H^2(\Omega;w)$.
Now, we proceed to prove assertion four.
Let the eigenvalues of $\mcA_\rho$ be ordered as
\begin{equation*}
\lambda_1\leq \lambda_2\leq \lambda_3\leq \cdots.
\end{equation*}
Assume that
\begin{equation*}
\lambda_{n - 1}<\lambda_n=\lambda_{n + 1}=\cdots=\lambda_{n + h - 1}<\lambda_{n + h},
\end{equation*}
indicating that $\lambda_n$ has multiplicity $h$. Let $I\subseteq\R$ be an interval such that its closure $\overline{I}$ contains only the eigenvalue $\lambda_n$ of $\mcA_\rho$. Then, there exists $\delta>0$ such that $I\subseteq (\lambda_{n - 1}+\delta, \lambda_{n + h}-\delta)$. Choose $T>0$ such that $|\tau|<T$ implies $|\rho(\tau)-\rho|_\infty<\delta$.
By Lemma \ref{06.26.L1} and for $|\tau|<T$, we have $|\lambda_{n - 1}(\tau)-\lambda_{n - 1}|<\delta$ and $|\lambda_{n + h}(\tau)-\lambda_{n + h}|<\delta$. Consequently, $\lambda_{n - 1}(\tau),\lambda_{n + h}(\tau)\notin I$. This implies that $\mcA+\rho(\tau)$ has at most $h$ eigenvalues in $I$ (counting multiplicity). However, as shown in the third assertion, it has at least $h$ eigenvalues $\lambda^{(1)}(\tau),\cdots, \lambda^{(h)}(\tau)$ in $I$. This completes the proof.
\end{proof}

Now, we are in position to prove our main result.

\begin{theorem}\label{06.26.T1}
Let $n\in\N$. Then $A_n$ is open in $L^\infty(\Omega)$.
\end{theorem}
\begin{proof}
Let $n\in\N$.
Denote
\begin{equation*}
A_n=\left\{\rho\in L^\infty(\Omega)\colon \text{the first $n$ eigenvalues of $\mcA+\rho$ are simple}\right\}.
\end{equation*}
Set $A_0=L^\infty(\Omega), A=\bigcap_{n=1}^\infty A_n$. Then,
\begin{equation*}
A\subseteq \cdots \subseteq A_n\subseteq \cdots \subseteq A_2\subseteq A_1\subseteq A_0.
\end{equation*}
We shall show that $A_n$ is open in $L^\infty(\Omega)$.
Let $\rho\in A_n$ be given and let
\begin{equation*}
\lambda_1<\cdots<\lambda_n<\lambda_{n+1}\leq \lambda_{n+2}\leq \lambda_{n+3}\leq \cdots
\end{equation*}
be the eigenvalues of $\mcA+\rho$. Then the first $n$ are simple. Denote
\begin{equation*}
\delta=\min\left\{\lambda_{j+1}-\lambda_j\colon j=1,\cdots, n\right\}. 
\end{equation*}
It is obvious that $\delta>0$. Denote
\begin{equation*}
U=\{\rho'\in L^\infty(\Omega)\colon |\rho'-\rho|_\infty<2^{-1}\delta\}.
\end{equation*}
For each $\rho'\in U$, denote $\{\lambda_j'\}_{j\in\N}$ be the eigenvalues of $\mcA+\rho'$ with
\begin{equation*}
\lambda_1'\leq \lambda_2'\leq \cdots.
\end{equation*}
Then, by Lemma \ref{06.26.L1}, we have
\begin{equation}\label{06.28.1}
|\lambda_j'-\lambda_j|\leq |\rho'-\rho|_\infty<2^{-1}\delta
\end{equation}
for all $j\in\N$.
Now, we prove $\lambda_j'\neq \lambda_{j+1}'$ for $j=1,\cdots, n$.
Note that from \eqref{06.28.1} we have
\begin{equation*}
\begin{split}
\delta
&\leq \lambda_{j+1}-\lambda_j\leq |\lambda_{j+1}'-\lambda_{j+1}|+|\lambda_j'-\lambda_{j+1}'|+|\lambda_j'-\lambda_j|<\delta+|\lambda_{j+1}'-\lambda_j'|. 
\end{split}
\end{equation*}
This implies that $\lambda_{j+1}'-\lambda_j'>0$ for all $j=1,\cdots, n$. Hence,  the first $n$ eigenvalues of $\mcA+\rho'$ are simple for all $\rho'\in U$. This proves the theorem.
\end{proof}

 \begin{theorem}\label{06.26.T2}
Let $n\in\N$. Then, $A_{n}$ is dense in $A_{n - 1}$, with $A_0 = L^\infty(\Omega)$.
\end{theorem}
\begin{proof} The proof is split into five steps.

{\it Step 1}. Assume that $\lambda$ is an eigenvalue of multiplicity $h$ for the operator $\mcA+\rho$. Let $\rho(\tau)=\rho+\tau\sigma$ be a perturbation of $\rho$, where $\sigma\in L^\infty(\Omega)$ is a given function. By Lemma \ref{06.24.L1}, we have
\begin{equation}\label{06.26.2}
\begin{split}
\lambda^{(j)}(\tau)=\lambda +\tau\alpha_j+\tau^2\beta_j(\tau),\quad u^{(j)}(\tau)=u_j+\tau v_j+\tau^2w_j(\tau)
\end{split}
\end{equation}
for all $j = 1,\cdots, h$ and for all $|\tau|\leq T_0$, where $v_j\in H_0^1(\Omega;w)\cap H^2(\Omega;w)$ for $j = 1,\cdots, h$.

{\it Step 2: Demonstrate the existence of $\sigma\in L^\infty(\Omega)$ such that for the perturbation $\rho(\tau)=\rho+\tau\sigma$, at least two of the $\alpha_j$ in \eqref{06.26.2} are distinct.}

Indeed, let $N=\ker(\mcA+\rho+\lambda)$. Then, $N$ is a vector space of dimension $h$ contained in $L^2(\Omega)$. Clearly,
\begin{equation*}
G_\sigma: N\times N\rightarrow \R, \ G_\sigma(f,g)=-\int_\Omega \sigma fg\,dx
\end{equation*}
is a bounded bilinear form, which induces a self-adjoint bounded linear operator
\begin{equation*}
G_\sigma: N\rightarrow N, \ \left(G_\sigma(f), g\right)_{L^2(\Omega)}=-\int_\Omega \sigma fg\,dx \text{ for all } f, g\in N.
\end{equation*}
Moreover, for any orthonormal basis $\{f_1,\cdots, f_h\}\subseteq N$, the matrix of $G_\sigma$ with respect to this basis has $(G_\sigma(f_j), f_k)_{L^2(\Omega)}$ as its $(j,k)$-th entry.
Now, for a given $\sigma$ and the orthonormal basis $\{u_1,\cdots, u_h\}$ of $N$, differentiate the equation
\begin{equation*}
\left(\mcA+\rho(\tau)+\lambda^{(j)}(\tau)\right) u^{(j)}(\tau)=0
\end{equation*}
with respect to $\tau$. Considering the forms of $\lambda^{(j)}(\tau)$ and $u^{(j)}(\tau)$ in \eqref{06.26.2}, we get
\begin{equation*}
(\mcA+\rho+\lambda)v_j+(\sigma+\alpha_j)u_j=0 \text{ for } j = 1,\cdots, h.
\end{equation*}
Since $\mcA+\rho+\lambda$ is a self-adjoint operator, it follows that $(\sigma+\alpha_j)u_j\bot N$ for $j = 1,\cdots, h$. Thus, $\int_\Omega (\sigma+\alpha_j)u_ju_k\,dx = 0$ for all $j = 1,\cdots, h$, which implies that
\begin{equation*}
-\int_\Omega \sigma u_ju_k\,dx=\alpha_j\delta_{jk} \text{ with } \delta_{jk}=
\begin{cases}
1, &\text{if } j = k,\\
0, &\text{if } j\neq k.
\end{cases}
\end{equation*}
Hence, the matrix of $G_\sigma$ is diagonal in this basis, i.e., it has the form $\diag(\alpha_1,\cdots,\alpha_n)$.
Again, for any fixed basis $\{f_1,\cdots, f_n\}$ of $N$, fix $j,k\in\{1,\cdots, h\}$ with $j\neq k$, and take $\sigma_0 = f_jf_k$ (note that $f_j,f_k\in L^\infty(\Omega)$). Then, the matrix of $G_{\sigma_0}$ has $-\|f_jf_k\|_{L^2(\Omega)}^2<0$ in its $(j,k)$-th entry. So, $G_{\sigma_0}$ is not a multiple of the identity, and thus it must have at least two distinct eigenvalues. By linear algebra, $G_{\sigma_0}$ has the same eigenvalues in $N$ under different bases. Hence, for the perturbation $\rho(\tau)=\rho+\tau\sigma_0$ with this $\sigma_0$, at least two of the $\alpha_j$ must be distinct.

{\it Step 3: Prove that $A_{n + 1}$ is dense in $A_n$ for all $n\in\N\cup\{0\}$.}

Let $\rho\in A_n$ and choose $\varepsilon_0>0$ small enough such that

(i) $|\rho'-\rho|_\infty<\varepsilon_0$ implies $\rho'\in A_n$; and

(ii) $\lambda_{n + 1}$ is the only eigenvalue of $\mcA+\rho$ in the interval $[\lambda_{n + 1}-\varepsilon _0,\lambda_{n + 1}+\varepsilon _0]$.

Denote $h$ as the multiplicity of $\lambda_{n + 1}$.

{\it Step 4: Show that if $h>1$, then for all $\varepsilon \in (0,\varepsilon _0)$, there exists $\rho'\in L^\infty(\Omega)$ such that $|\rho'-\rho|_\infty<\varepsilon $ and the multiplicity of $\lambda_n'$ is $<h$. }

If {\it Step 4} holds, by applying {\it Step 4} repeatedly, we can obtain a sequence $\rho_1,\cdots, \rho_h$ such that $|\rho_{j + 1}-\rho_j|_\infty<\frac{\varepsilon }{h}$ for $j = 1,\cdots, h - 1$, and the $(n + 1)$th eigenvalue of $\mcA+\rho_j$ has multiplicity $\leq h - j + 1$. Then $|\rho_h-\rho|_\infty<\varepsilon $ and the first $n + 1$ eigenvalues of $\mcA+\rho_h$ are simple. Thus, we have completed the proof of the theorem.

{\it Step 5}. Now, we prove {\it Step 4.}

By {\it Step 2}, choose $\sigma\in L^\infty(\Omega)$ such that at least two of the $\alpha_j$ in \eqref{06.26.2} are distinct. That is, there exists $T_1>0$ such that for $0<|\tau|<T_1$, at least two of the $\lambda^{(j)}(\tau)$ are distinct.
Since $\lambda:=\lambda_n$ is the only eigenvalue of $\mcA+\rho$ in $[\lambda-\varepsilon , \lambda+\varepsilon ]$, by the fourth assertion of Lemma \ref{06.24.L1}, there exists $T_2>0$ such that for all $0<|\tau|<T_2$, $\lambda^{(1)}(\tau),\cdots, \lambda^{(h)}(\tau)$ are exactly the $h$ eigenvalues of $\mcA+\rho(\tau)$ in $(\lambda-\varepsilon , \lambda+\varepsilon )$ (counting multiplicity). Moreover, when $|\tau|<T_1$, at least two of these eigenvalues are distinct, so they must all have multiplicity $<h$.
Finally, set $\tau=\frac{1}{2}\min\{T_1, T_2, |\sigma|_\infty^{-1}\varepsilon\}$ and define $\rho'=\rho(\tau)=\rho+\tau\sigma$. Then, we have $|\rho' - \rho|_\infty=|\tau||\sigma|_\infty<\varepsilon$. Moreover, by Lemma \ref{06.26.L1}, it follows that $|\lambda_n' - \lambda_n|<\varepsilon$. From the above analysis, we can conclude that $\lambda_n'$ must coincide with one of the $\lambda^{(j)}(\tau)$ and have a multiplicity less than $h$.
\end{proof}

Recall that a subset of a Banach space is called residual if it can be written as a countable intersection of open dense sets. By combining Theorems \ref{06.26.T1} and \ref{06.26.T2}, we immediately obtain our final result, namely Theorem \ref{06.26.T3}, as presented below.

\begin{theorem}\label{06.26.T3}
The set
\begin{equation*}
\left\{\rho \in L^\infty(\Omega) \colon \mathcal{A} + \rho \text{ has simple eigenvalues}\right\}
\end{equation*}
contains a residual subset within $(L^\infty(\Omega), |\cdot|_\infty)$.
\end{theorem}

\noindent\bf{Concluding remarks}.  Several important problems remain open and may serve as directions for future research. For example:

(1)	In Section \ref{S3}, for the nodal line of the second eigenfunction of equation \eqref{01.08.1} with weight $w = |x - x_0|^\alpha$ for $0 < \alpha < 2$ and $x_0 \in \partial\Omega \subset \mathbb{R}^2$, we know the nodal line divides the domain $\Omega$ into two subdomains, it remains unclear whether the nodal line passes through the   point $x_0$.

(2)	Also in Section \ref{S3}, how can one estimate the lower and upper bounds (see \cite{Yau}) for the Hausdorff measure of the nodal set of the eigenfunctions of \eqref{01.08.1}? In particular, do the constants appearing in these bounds depend only on the $A_p$ constant of the weight $w$?

(3)	In Section \ref{S4}, consider replacing the operator $\mathcal{A}$ with a perturbed operator $\mathcal{A}_\varepsilon$ defined by $\mathcal{A}_\varepsilon u = \operatorname{div}(A_\varepsilon \nabla u)$, where $A_\varepsilon \in \mathbb{R}^{N \times N}$ and $\varepsilon > 0$ represents a perturbation of $w$. Under this perturbation, does an analogue of Theorem \ref{06.26.T3} still hold?

(4)	In Section \ref{S4}, what occurs when the domain $\Omega$ is perturbed to a nearby domain $\Omega_\varepsilon$? That is, if $\Omega_\varepsilon$ is a suitable perturbation of $\Omega$, can we still obtain a conclusion analogous to Theorem \ref{06.26.T3}?

(5)	Finally, what can be said when $\Omega$ is replaced by a smooth, compact manifold of dimension $N$ with or without boundary? How do the main results of this paper extend to such a setting?
 

\end{document}